\documentclass[11pt]{amsart}
\usepackage[T1]{fontenc}
\usepackage[utf8]{inputenc}
\usepackage{amsmath,amssymb,amsthm,amsxtra,mathtools}
\usepackage[varg]{txfonts}
\usepackage{microtype}
\usepackage[dvipsnames,svgnames,table]{xcolor}
\usepackage{aliascnt}
\usepackage[linktocpage=true,colorlinks=true,linkcolor=Blue,
  citecolor=BrickRed,urlcolor=RoyalBlue]{hyperref}
\usepackage[nameinlink,noabbrev]{cleveref}

\hypersetup{
  pdfauthor={Thialita M. Nascimento and Eduardo V. Teixeira},
  pdftitle={Sharp  Hessian integrability for fully nonlinear elliptic supersolutions in low dimensions}
}

\numberwithin{equation}{section}

\theoremstyle{plain}

\newtheorem{theorem}{Theorem}[section]

\newaliascnt{proposition}{theorem}
\newtheorem{proposition}[proposition]{Proposition}
\aliascntresetthe{proposition}

\newaliascnt{lemma}{theorem}
\newtheorem{lemma}[lemma]{Lemma}
\aliascntresetthe{lemma}

\newaliascnt{corollary}{theorem}

\aliascntresetthe{corollary}

\theoremstyle{definition}

\newaliascnt{definition}{theorem}

\aliascntresetthe{definition}

\theoremstyle{remark}

\newaliascnt{remark}{theorem}
\newtheorem{remark}[remark]{Remark}
\aliascntresetthe{remark}

\crefname{theorem}{theorem}{theorems}
\Crefname{theorem}{Theorem}{Theorems}
\crefname{proposition}{proposition}{propositions}
\Crefname{proposition}{Proposition}{Propositions}
\crefname{lemma}{lemma}{lemmas}
\Crefname{lemma}{Lemma}{Lemmas}
\crefname{corollary}{corollary}{corollaries}
\Crefname{corollary}{Corollary}{Corollaries}
\crefname{definition}{definition}{definitions}
\Crefname{definition}{Definition}{Definitions}
\crefname{remark}{remark}{remarks}
\Crefname{remark}{Remark}{Remarks}

\newcommand{\Rn}{\mathbb R^n}

\newcommand{\norm}[1]{\left\lVert #1\right\rVert}
\newcommand{\abs}[1]{\left\lvert #1\right\rvert}

\newcommand{\Mminus}{\mathcal M^-_{\lambda,\Lambda}}
\newcommand{\Mplus}{\mathcal M^+_{\lambda,\Lambda}}
\newcommand{\ThetaU}{\underline\Theta}
\newcommand{\eps}{\varepsilon}
\newcommand{\dd}{\,\mathrm d}

\newcommand{\tr}{\operatorname{tr}}
\newcommand{\indm}{\operatorname{ind}_{-}}

\title[Sharp Hessian integrability in low dimensions]
{Sharp Hessian integrability for fully nonlinear elliptic supersolutions in low dimensions}

\author[T. M. Nascimento]{Thialita M. Nascimento}
\address[T. M. Nascimento]{Department of Mathematics, Universidade Federal da Para\'iba, Jo\~ao Pessoa, PB, Brazil}
\email{tnascimento@mat.ufpb.br}

\author[E. V. Teixeira]{Eduardo V. Teixeira}
\address[E. V. Teixeira]{Department of Mathematics, Oklahoma State University, Stillwater, Oklahoma, USA}
\email{eduardo.teixeira@okstate.edu}

\subjclass[2020]{35J60, 35B65, 35D40}
\keywords{Fully nonlinear elliptic equations, viscosity supersolutions, Hessian integrability, contact geometry, negative index, Pucci operators}

\begin{document}

\begin{abstract}
We settle the planar sharp Hessian-integrability conjecture of Armstrong, Silvestre, and Smart [\emph{Comm. Pure Appl. Math.} \textbf{65} (2012), 1169--1184] for viscosity supersolutions of fully nonlinear uniformly elliptic equations.  If $\kappa=\Lambda/\lambda$, then the optimal exponent in the $W^{2,\eps}$ regularity theory in the plane is exactly
\[
  \eps_2(\kappa)=\frac{2}{\kappa+1}.
\]
The same mechanism reaches the known upper obstruction in dimension three throughout the full range $1\le\kappa\le4$ and therefore gives
\[
  \eps_3(\kappa)=\frac{3}{2\kappa+1}
  \qquad(1\le\kappa\le4).
\]
Beyond this threshold, it yields a closed algebraic lower bound and identifies the precise spectral configuration responsible for the remaining gap.
The proof introduces a spectrally resolved continuum-in-opening mechanism
for contact geometry.  Throughout the contact evolution, the vertex
Jacobian retains the negative index of the Hessian and its mean negative
curvature---data erased by the classical pointwise reduction---and these
variables are optimized only after integration over all openings.  This
reduces the analysis to a finite family of explicit one-dimensional
kernels whose critical exponents recover the sharp planar threshold,
identify the exact three-dimensional regime above, and expose a broader
principle for extracting second-order regularity from the spectral
geometry of contact sets.
\end{abstract}

\maketitle

\section{Introduction}

\subsection{The sharp exponent problem}

The one-sided $W^{2,\eps}$ estimate controls the distribution of the least
opening of a paraboloid supporting a supersolution from below.  It was first
established for strong solutions of linear nondivergence equations by Lin
\cite[Theorem, p.~448]{Lin}.  Caffarelli subsequently extended the estimate
to viscosity solutions \cite[Corollary~3]{Caffarelli}; see also
\cite[Proposition~7.4 and Lemma~7.8]{CaffarelliCabre}.  The exponent also
enters directly into the partial-regularity theorem of Armstrong, Silvestre,
and Smart: it provides the corresponding lower bound for the codimension of
the singular set
\cite[Theorem~1 and Remark~5.4]{ArmstrongSilvestreSmart}.  Determining its
optimal value is therefore a basic structural problem in the regularity
theory of uniformly elliptic equations.

A polynomial lower bound for the exponent in terms of the ellipticity ratio
was obtained by Le \cite{Le}. Mooney subsequently developed a
global convex-analytic method, based on sliding paraboloids and the area
formula, which improved this dependence to the optimal order in the planar
case \cite{Mooney}. The quantitative bound in higher
dimensions was later sharpened in by the authors in
\cite{NascimentoTeixeiraJMPA}. Despite these
substantial advances, the exact threshold remained unknown even in the
planar case: it was unclear whether the radial-profile obstruction \cite{ArmstrongSilvestreSmart} was decisive or
whether genuinely nonradial contact geometry could force a smaller
exponent.

To place this question in its natural universal setting, let
$0<\lambda\le\Lambda$, set
\[
\kappa:=\frac{\Lambda}{\lambda},
\]
and consider the universal class
\begin{equation}\label{eq:Pucci-supersolution-intro}
\Mminus(D^2u)\le0
\qquad\text{in }B_1\subset\Rn.
\end{equation}
Here
\[
 \mathcal M^-_{\lambda,\Lambda}(M)
 =\lambda\sum_{\mu_i(M)>0}\mu_i(M)
  +\Lambda\sum_{\mu_i(M)<0}\mu_i(M)
\]
is the lower Pucci extremal operator.  We use the viscosity convention of \cite[Definition~2.2 and Remarks~2.3]{CIL}.  Every viscosity supersolution of a $(\lambda,\Lambda)$-elliptic equation $F(x,D^2u)=0$, normalized by $F(x,0)=0$, satisfies \eqref{eq:Pucci-supersolution-intro}; hence estimates for this class are universal within the context of fully nonlinear elliptic problems.

For a bounded domain $\Omega$ and $x\in\Omega$, let $\ThetaU(u,\Omega)(x)$ be the least $A\ge0$ for which a paraboloid of opening $-A$ touches $u$ from below at $x$ throughout $\Omega$.  Given $q>0$, let $\mathsf E_q(n,\kappa)$ denote the assertion that
\begin{equation}\label{eq:Eq-definition}
 \int_{B_{1/2}}\ThetaU(u,B_1)^q\dd x
 \le C(n,\kappa,q)\norm{u}_{L^\infty(B_1)}^q
\end{equation}
for every $u\in C(\overline B_1)$ satisfying
\eqref{eq:Pucci-supersolution-intro}, and define
\begin{equation}\label{eq:def-sharp-threshold}
 \eps_n(\kappa):=\sup\big\{q>0:\mathsf E_q(n,\kappa)
 \text{ holds} \big\}.
\end{equation}
Throughout the paper, $\abs{M}:=\bigl(\tr(M^2)\bigr)^{1/2}$ denotes the
Hilbert--Schmidt norm of a symmetric matrix $M$.  At a point of twice
differentiability at which $\ThetaU(u,B_1)$ is finite, each negative
eigenvalue of $D^2u$ has magnitude at most $\ThetaU(u,B_1)$.  The Pucci
inequality then bounds the sum of the positive eigenvalues by $\kappa$
times the total negative curvature.  Consequently,
\begin{equation}\label{eq:Hessian-controlled-by-Theta}
 \abs{D^2u(x)}
 \le C(n,\kappa)\ThetaU(u,B_1)(x).
\end{equation}
Thus \eqref{eq:def-sharp-threshold} is the natural universal one-sided
Hessian-integrability threshold.

Armstrong, Silvestre, and Smart in \cite{ArmstrongSilvestreSmart} constructed a planar example showing that
\[
  \eps_2(\kappa)\le\frac{2}{\kappa+1}
\]
and conjectured equality; see \cite[Remark~3.3 and Conjecture~3.4]{ArmstrongSilvestreSmart}.  In dimensions $n\ge3$, the construction assembled from translated radial profiles in \cite{NascimentoTeixeiraJMPA} gives the sharper obstruction
\begin{equation}\label{eq:radial-upper-intro}
  \eps_n(\kappa)\le\frac{n}{(n-1)\kappa+1}.
\end{equation}
Our later planar estimate reached a fixed proportion of the conjectured
value, uniformly as $\kappa\to\infty$; see
\cite[Theorem~1]{NascimentoTeixeiraJDE}. The sharp endpoint nevertheless
remained out of reach, leaving the conjecture unresolved even in its first
nontrivial dimension.

\subsection{The loss hidden in the contact argument}

Sliding paraboloids from a convex envelope yields a vertex map with an
exact Jacobian.  At each new contact point, this Jacobian records the entry
opening, the negative index
\[
  r=\indm(D^2u),
\]
and the mean magnitude of the negative eigenvalues.  Here $\indm(M)$
denotes the number of strictly negative eigenvalues of $M$, counted with
multiplicity.  The usual pointwise optimization discards the coupling
between this spectral information and the opening at which contact occurs.

We introduce a spectrally resolved continuum-in-opening method that treats
the nested contact sets as a single geometric evolution.  The full contact
state---entry opening, negative index, and mean negative curvature---is
carried through the area formula and preserved until the evolution has
been integrated.  The essential innovation is this persistence of the
spectral data, not merely their subsequent optimization.  Since a new
contact layer may contain several negative indices, we stratify it by the
actual value of $r$ before applying the area formula.  Each stratum then
produces its own sharp determinant law and explicit one-dimensional
aperture kernel.  Previous arguments collapse these strata into a single global index bound,
thereby irreversibly discarding the spectral information needed to reach
the sharp low-dimensional thresholds.

\subsection{The planar theorem and a sharp three-dimensional regime}

Our first theorem settles the conjecture of Armstrong, Silvestre, and Smart.

\begin{theorem}[Sharp planar threshold]\label{thm:planar-intro}
For every $\kappa\ge1$,
\begin{equation}\label{eq:sharp-planar-intro}
{\displaystyle \eps_2(\kappa)=\frac{2}{\kappa+1}.}
\end{equation}
\end{theorem}

The same mechanism produces a second sharp result in dimension three.  This is not merely a perturbative statement near the Laplacian: the lower theory reaches the radial-profile obstruction throughout the entire interval $1\le\kappa\le4$.  Beyond that threshold, the two possible negative indices exchange order and yield an explicit remaining gap.

\begin{theorem}[Dimension three]\label{thm:three-dimensional-intro}
For every $\kappa\ge1$, the estimate $\mathsf E_p(3,\kappa)$ holds whenever $0<p<\underline p_3(\kappa)$, where
\begin{equation}\label{eq:3d-lower-intro}
 \underline p_3(\kappa)=
 \begin{cases}
 \displaystyle \frac{3}{2\kappa+1},&1\le\kappa\le4,\\[1.2ex]
 \displaystyle \frac{\kappa+10-\sqrt{\kappa^2+20\kappa+4}}{2(\kappa+2)},&\kappa>4.
 \end{cases}
\end{equation}
Consequently,
\begin{equation}\label{eq:3d-exact-intro}
{\displaystyle
  \eps_3(\kappa)=\frac{3}{2\kappa+1}
  \quad\text{for }1\le\kappa\le4.}
\end{equation}
For $\kappa>4$,
\begin{equation}\label{eq:3d-gap-intro}
 \frac{\kappa+10-\sqrt{\kappa^2+20\kappa+4}}{2(\kappa+2)}
 \le\eps_3(\kappa)
 \le\frac{3}{2\kappa+1}.
\end{equation}
\end{theorem}

The upper bounds in \Cref{thm:planar-intro,thm:three-dimensional-intro} are the examples from \cite{ArmstrongSilvestreSmart,NascimentoTeixeiraJMPA}.  The present paper supplies the matching lower theory: complete in the plane and sharp in dimension three over a substantial ellipticity range.  The value $4$ is intrinsic to the three-dimensional Pucci cone and not an artifact of the optimization; see \Cref{rem:kappa-four}.

\subsection{The general mechanism}

The proof extends, without changing its principle, to every dimension.  For $1\le r\le n-1$, define
\begin{equation}\label{eq:index-kernel-intro}
 P_{n,r,\kappa}(z)
 :=(1-z)^r\left(1+\frac{\kappa r}{n-r}z\right)^{n-r},
 \qquad 0\le z\le1.
\end{equation}
If $\kappa>1$, let $p_{n,r}(\kappa)\in(0,1)$ be the unique solution of
\begin{equation}\label{eq:critical-equation-intro}
  \int_0^1 z^{-p_{n,r}(\kappa)}P'_{n,r,\kappa}(z)\dd z=0.
\end{equation}
For $\kappa=1$, set $p_{n,r}(1)=1$, and put
\begin{equation}\label{eq:underline-p-intro}
  \underline p_n(\kappa):=\min_{1\le r\le n-1}p_{n,r}(\kappa).
\end{equation}
For each fixed $n$, these exponents are determined by an explicit
algebraic equation.  Indeed, if
\[
  P_{n,r,\kappa}(z)=1+\sum_{j=1}^n a_jz^j,
\]
then \eqref{eq:critical-equation-intro} is equivalent to
\begin{equation}\label{eq:algebraic-critical-intro}
  p\sum_{j=1}^n\frac{a_j}{j-p}=1.
\end{equation}

\begin{theorem}[General lower estimate]\label{thm:general-lower-intro}
Let $n\ge2$, $\kappa=\Lambda/\lambda\ge1$, and let $u\in C(\overline B_1)$ satisfy \eqref{eq:Pucci-supersolution-intro}.  For every
\[
  0<p<\underline p_n(\kappa)
\]
there exists $C=C(n,\kappa,p)$ such that
\begin{equation}\label{eq:main-general-estimate-intro}
 \int_{B_{1/2}}\ThetaU(u,B_1)^p\dd x
 \le C\norm{u}_{L^\infty(B_1)}^p.
\end{equation}
In particular,
\[
  \eps_n(\kappa)\ge\underline p_n(\kappa).
\]
\end{theorem}

The integer $r$ in \eqref{eq:index-kernel-intro} is the actual number of negative Hessian eigenvalues at the contact point.  Each value of $r$ produces one explicit kernel, and the smallest critical exponent governs the universal estimate.  When $r=n-1$, the critical value is exactly
\[
  p_{n,n-1}(\kappa)=\frac{n}{(n-1)\kappa+1},
\]
the radial-profile scale in \eqref{eq:radial-upper-intro}.
The range in which the radial branch governs the kernel family is
determined in \Cref{prop:radial-kernel-threshold}.
The final section also records the fixed-dimensional high-contrast
consequence.
It improves the leading constant in the lower bound of \cite{NascimentoTeixeiraJMPA}, while leaving the power $\kappa^{-(n-1)}$ unchanged.

\subsection{Organization of the paper}

\Cref{sec:contact-geometry} develops the contact geometry and the two compactness devices used in the proof: localization and inf-convolution.  In \Cref{sec:spectral-jacobian} we derive the exact contact Jacobian and separate the new contacts according to the number of negative eigenvalues.  The continuum integration in the opening, including the terminal cutoff, is carried out in \Cref{sec:continuum}.  The sharp planar theorem and the three-dimensional calculation are proved in \Cref{sec:low-dimensions}.  We conclude in \Cref{sec:higher-dimensions} with index comparisons and high-contrast asymptotics that connect the general kernels to the earlier higher-dimensional theory.

\section{Contact geometry, localization, and stability}\label{sec:contact-geometry}

\subsection{Paraboloid contact sets}

Let $\Omega\subset\Rn$ be bounded and strictly convex.  For $a\ge0$ and $u\in C(\overline\Omega)$, define the $a$-convex envelope
\begin{equation}\label{eq:a-envelope}
 \Gamma_u^a(x):=
 \sup\left\{-\frac a2\abs{x}^2+L(x):
 -\frac a2\abs{y}^2+L(y)\le u(y)\text{ in }\overline\Omega,
 \ L\text{ affine}\right\}.
\end{equation}
The corresponding contact set is
\begin{equation}\label{eq:contact-set}
 A_a(u,\Omega):=\{x\in\overline\Omega:u(x)=\Gamma_u^a(x)\}.
\end{equation}
We omit $\Omega$ when the domain is fixed.  Strict convexity implies that every boundary point belongs to $A_a(u,\Omega)$.  At an interior point, membership in $A_a(u,\Omega)$ is equivalent to the existence of a paraboloid of opening $-a$ touching $u$ from below throughout $\Omega$.  The sets are nested:
\[
 A_a(u,\Omega)\subset A_b(u,\Omega)
 \qquad(0\le a\le b).
\]
Moreover,
\begin{equation}\label{eq:Theta-definition}
 \ThetaU(u,\Omega)(x)=\inf\{a>0:x\in A_a(u,\Omega)\},
\end{equation}
with the value $+\infty$ if the set is empty.

We shall use repeatedly the elementary transformation
\begin{equation}\label{eq:envelope-transform}
 \Gamma_{\beta u+\frac\gamma2\abs{x}^2}^{a}
 =\beta\Gamma_u^{(a+\gamma)/\beta}+\frac\gamma2\abs{x}^2,
 \qquad a\ge0,\qquad \beta>0,\qquad a+\gamma\ge0.
\end{equation}

\begin{lemma}[Separation of new contacts]\label{lem:separation}
Let $v\in C(\overline\Omega)$ satisfy
\[
 \Mminus(D^2v)\le K<\infty
 \qquad\text{in }\Omega.
\]
Suppose that a paraboloid $Q$ of opening $-\delta<0$ touches
$\Gamma_v^0$ from below at
\[
 x_0\in\Omega\setminus A_0(v,\Omega).
\]
Slide $Q$ upward until $Q+t$, with $t>0$, first touches $v$.  If a
first-contact point $x_1$ lies in $\Omega$, then
\[
 x_1\notin A_0(v,\Omega).
\]
\end{lemma}

\begin{proof}
This is the interior assertion of \cite[Lemma~3.1]{Mooney}.  Suppose, to
the contrary, that $x_1\in A_0(v,\Omega)$, and let $L$ be an affine
support of $\Gamma_v^0$ at $x_1$.  Thus
\[
 L\le\Gamma_v^0\le v
 \quad\text{in }\overline\Omega,
 \qquad
 L(x_1)=\Gamma_v^0(x_1)=v(x_1).
\]
If $Q+t$ is tangent to $L$ at $x_1$, then the strict concavity of
$Q+t-L$ gives $Q+t\le L$ everywhere.  Evaluating at $x_0$, where
$Q(x_0)=\Gamma_v^0(x_0)$, yields
\[
 Q(x_0)+t\le L(x_0)\le\Gamma_v^0(x_0)=Q(x_0),
\]
contrary to $t>0$.

Hence $Q+t$ and $L$ meet transversally at $x_1$.  After subtracting
$L$, translating and rotating the coordinates, and applying fixed
positive rescalings in the independent and dependent variables, we may
assume that
\[
 \max\{0,1-\abs{x}^2\}
\]
touches the normalized function from below at $e_1$.  For every $A>0$,
the function
\[
 \phi_A(x):=1-\abs{x}+A(\abs{x}-1)^2
\]
then touches it from below in a neighborhood of $e_1$.  At that point,
$D^2\phi_A$ has radial eigenvalue $2A$ and $n-1$ tangential eigenvalues
equal to $-1$.  Consequently,
\[
 \Mminus(D^2\phi_A(e_1))
 =2\lambda A-(n-1)\Lambda.
\]
The preceding normalization changes the finite right-hand side of the
viscosity inequality only by a fixed factor.  Letting $A\to\infty$
therefore gives a contradiction.
\end{proof}

\begin{lemma}[Expansion of transformed vertices]\label{lem:vertex-expansion}
Let $w$ be convex in $\Omega$, let $\delta>0$, and let $F\subset\Omega$ be compact.  Denote by $V_F$ the set of vertices of all paraboloids of opening $-\delta$ tangent from below to $w$ at points of $F$.  Then
\begin{equation}\label{eq:vertex-expansion}
 \abs{V_F}\ge\abs{F}.
\end{equation}
\end{lemma}

\begin{proof}
A paraboloid of opening $-\delta$ tangent to $w$ at $x$ has vertex
\[
 x+\frac1\delta p,
 \qquad p\in\partial w(x).
\]
Thus $V_F=\partial h(F)$ for
\[
 h(x)=\frac12\abs{x}^2+\frac1\delta w(x).
\]
The function $h$ is convex and its Alexandrov Hessian satisfies $D^2h\ge I$.  Its Monge--Amp\`ere measure therefore gives
\[
 \abs{V_F}=\abs{\partial h(F)}
 \ge\int_F\det D^2h\dd x
 \ge\abs{F}.
\]
This is the argument of \cite[Lemma~3.2]{Mooney}.
\end{proof}

\subsection{Localization}

The iteration is cleanest when all noncontact sets are compactly contained in the region where the differential inequality is available.

\begin{lemma}[Localization by a concave cap]\label{lem:localization}
There are numerical constants $0<\eta<1$, $R>2$, and $a_0>1$ with the following property.  If $u\in C(\overline B_1)$ satisfies \eqref{eq:Pucci-supersolution-intro} and $\norm{u}_{L^\infty(B_1)}>0$, then there are a positive affine rescaling $v=c_0u+c_1$ and a continuous supersolution $\widetilde v$ in $B_R$ such that
\begin{enumerate}
\item $1\le v\le1+\eta$ in $B_1$;
\item $\widetilde v=v$ in $B_{3/4}$ and $\widetilde v\le v$ in $B_1$;
\item $B_R\setminus A_a(\widetilde v,B_R)\Subset B_1$ for every $a\ge a_0$;
\item
\[
 \ThetaU(v,B_1)(x)\le\ThetaU(\widetilde v,B_R)(x)
 \qquad(x\in B_{1/2}).
\]
\end{enumerate}
The constants are independent of $u$, $\lambda$, and $\Lambda$.
\end{lemma}

\begin{proof}
Take $\eta=1/10$, fix $R=3$, and put
\[
 v(x)=1+\eta\frac{u(x)+\norm{u}_{L^\infty(B_1)}}{2\norm{u}_{L^\infty(B_1)}}.
\]
Then $1\le v\le1+\eta$.  Let $q(x)=3(1-\abs{x}^2)$.  Since $q>1+\eta$ in $B_{3/4}$ and $q<1$ for $\abs{x}>\sqrt{2/3}$, define
\[
 \widetilde v(x)=
 \begin{cases}
  \min\{v(x),q(x)\},&x\in B_1,\\
  q(x),&x\in B_R\setminus B_1.
 \end{cases}
\]
Near $\partial B_1$ the minimum equals $q$, so the definition is continuous.  Both $v$ and $q$ are supersolutions, and the minimum of two viscosity supersolutions is a supersolution.  The first two assertions follow.

Set
\[
 K=\{x\in B_1:\widetilde v(x)<q(x)\}.
\]
Then $K\subset B_{\sqrt{2/3}}$ and $0\le q-\widetilde v\le2$ on $K$.  Choose $r_1\in(\sqrt{2/3},1)$ and then $a_0>6$ so that
\[
 \left(\frac{a_0}{2}-3\right)
 (r_1-\sqrt{2/3})^2\ge2.
\]
For $\abs{x}\ge r_1$, the paraboloid
\[
 P_x(y)=q(x)+\nabla q(x)\cdot(y-x)-\frac{a_0}{2}\abs{y-x}^2
       =q(y)-\left(\frac{a_0}{2}-3\right)\abs{y-x}^2
\]
lies below $\widetilde v$ in $B_R$ and agrees with it at $x$.  Hence
\[
 B_R\setminus A_{a_0}(\widetilde v,B_R)
 \subset B_{r_1}\Subset B_1,
\]
and the nesting of the contact sets proves the third assertion.

Finally, a paraboloid touching $\widetilde v$ from below at $x\in B_{1/2}$ lies below $v$ in $B_1$, and the two functions agree at $x$.  This gives the last assertion.
\end{proof}

\subsection{Inf-convolution stability of the contact opening}

The spectral analysis will first be carried out for semiconcave supersolutions.  The next lemma contains the entire passage back to continuous viscosity supersolutions.

\begin{lemma}[Inf-convolution stability]\label{lem:inf-convolution}
Let $u\in C(\overline B_1)$ satisfy
\eqref{eq:Pucci-supersolution-intro}, and assume that
\[
 M:=\norm{u}_{L^\infty(B_1)}>0.
\]
For $j\in \mathbb{N}$, define
\begin{equation}\label{eq:inf-convolution}
 u_j(x):=
 \inf_{y\in\overline B_1}
 \bigl\{u(y)+j\abs{x-y}^2\bigr\}.
\end{equation}
After discarding finitely many indices, set
\[
 \delta_j:=\left(\frac{2M}{j}\right)^{1/2},
 \qquad
 R_j:=1-\delta_j.
\]
Then $R_j\uparrow1$, and the following properties hold:
\begin{enumerate}
\item $u_j$ is semiconcave,
\[
 \norm{u_j}_{L^\infty(B_1)}\le M,
 \qquad
 u_j\longrightarrow u
 \quad\text{uniformly in }\overline B_1;
\]
\item $u_j$ is a viscosity supersolution of
\eqref{eq:Pucci-supersolution-intro} in $B_{R_j}$;
\item for every $x\in B_{1/2}$,
\begin{equation}\label{eq:Theta-lower-semicontinuity}
 \ThetaU(u,B_1)(x)
 \le
 \liminf_{j\to\infty}
 \ThetaU(u_j,B_{R_j})(x).
\end{equation}
\end{enumerate}
\end{lemma}

\begin{proof}
Let $y_j(x)\in\overline B_1$ be a minimizer in
\eqref{eq:inf-convolution}.  Comparing with the competitor $y=x$ gives
\[
 j\abs{x-y_j(x)}^2
 \le u(x)-u(y_j(x))
 \le 2M,
\]
and hence
\begin{equation}\label{eq:inf-convolution-displacement}
 \abs{x-y_j(x)}\le\delta_j.
\end{equation}
In particular, if $x\in B_{R_j}$, every minimizing point $y_j(x)$ lies
in $B_1$.  The standard inf-convolution lemma therefore implies that
$u_j$ is semiconcave and that lower viscosity jets of $u_j$ at points
of $B_{R_j}$ transfer to lower viscosity jets of $u$ at the
corresponding minimizing points.  Consequently, $u_j$ satisfies
\eqref{eq:Pucci-supersolution-intro} in $B_{R_j}$; see
\cite[Lemma~A.5]{CIL}.

For completeness, the semiconcavity is also immediate from
\[
 u_j(x)-j\abs{x}^2
 =
 \inf_{y\in\overline B_1}
 \bigl\{u(y)+j\abs{y}^2-2j x\cdot y\bigr\},
\]
since the right-hand side is an infimum of affine functions of $x$.
Moreover,
\[
 -M\le u_j\le u\le M.
\]
If $\omega_u$ is the modulus of continuity of $u$ on
$\overline B_1$, then \eqref{eq:inf-convolution-displacement} gives
\[
 u(x)-\omega_u(\delta_j)
 \le u_j(x)\le u(x)
 \qquad(x\in\overline B_1).
\]
Thus $u_j\to u$ uniformly in $\overline B_1$.

It remains to prove \eqref{eq:Theta-lower-semicontinuity}.  Fix
$x\in B_{1/2}$ and set
\[
 \ell:=
 \liminf_{j\to\infty}
 \ThetaU(u_j,B_{R_j})(x).
\]
There is nothing to prove if $\ell=+\infty$.  Otherwise, pass to a
subsequence, not relabeled, along which the lower limit is attained.
Choose a supporting paraboloid
\[
 P_j(y)
 =
 u_j(x)+\xi_j\cdot(y-x)
 -\frac{A_j}{2}\abs{y-x}^2
\]
for $u_j$ in $B_{R_j}$ such that
\[
 A_j\le
 \ThetaU(u_j,B_{R_j})(x)+\frac1j.
\]
The sequence $(A_j)$ is bounded.  Since $R_j\uparrow1$ and
$x\in B_{1/2}$, we may fix $h>0$ such that
\[
 x\pm he\in B_{R_j}
\]
for every unit vector $e$ and all sufficiently large $j$.  Evaluating
$P_j\le u_j$ at these two points and using
$\norm{u_j}_{L^\infty(B_1)}\le M$ shows that $(\xi_j)$ is bounded.
After passing to a further subsequence,
\[
 A_j\longrightarrow A,
 \qquad
 \xi_j\longrightarrow\xi,
 \qquad
 A=\ell.
\]
The paraboloids $P_j$ therefore converge locally uniformly to
\[
 P(y)
 =
 u(x)+\xi\cdot(y-x)
 -\frac A2\abs{y-x}^2.
\]

For every fixed $y\in B_1$, one has $y\in B_{R_j}$ for all sufficiently
large $j$.  Passing to the limit in
\[
 P_j(y)\le u_j(y)
\]
and using the uniform convergence of $u_j$ gives
\[
 P(y)\le u(y)
 \qquad(y\in B_1),
 \qquad
 P(x)=u(x).
\]
By continuity, the inequality extends to $\overline B_1$.  Thus $P$
supports $u$ from below at $x$, and consequently
\[
 \ThetaU(u,B_1)(x)\le A=\ell.
\]
This proves \eqref{eq:Theta-lower-semicontinuity}.
\end{proof}

\section{The contact Jacobian and spectral stratification}\label{sec:spectral-jacobian}

Let $\Omega\subset\Rn$ be bounded and strictly convex, and let
$u\in C(\overline\Omega)$ be semiconcave in $\Omega$.  Assume that
\begin{equation}\label{eq:Pucci-on-Omega}
 \Mminus(D^2u)\le0
 \qquad\text{in }\Omega,
\end{equation}
and that, for some $a_0>0$,
\begin{equation}\label{eq:compact-noncontact}
 \Omega\setminus A_a(u,\Omega)\Subset\Omega
 \qquad\text{for every }a\ge a_0.
\end{equation}

At every point $x$ where $u$ is twice differentiable in the sense of
Alexandrov, the viscosity inequality holds with the Alexandrov Hessian
$D^2u(x)$.  Define
\[
 r(x):=\indm(D^2u(x)).
\]
When $1\le r(x)\le n$, write the strictly negative eigenvalues of
$D^2u(x)$ as
\[
 -q_1(x),\ldots,-q_{r(x)}(x),
 \qquad q_i(x)>0,
\]
and its remaining eigenvalues as
\[
 \mu_1(x),\ldots,\mu_{n-r(x)}(x),
 \qquad \mu_j(x)\ge0.
\]
The mean negative curvature is
\begin{equation}\label{eq:mean-negative-curvature}
 s(x):=\frac{1}{r(x)}
 \sum_{i=1}^{r(x)}q_i(x)
 \qquad\text{if }r(x)\ge1.
\end{equation}
If $r(x)=0$, then \eqref{eq:Pucci-on-Omega} forces $D^2u(x)=0$, and we
set $s(x)=0$.  The Alexandrov Hessian is measurable on its domain of
definition.  Since the ordered eigenvalues are continuous functions of
a symmetric matrix, both $r$ and $s$ are measurable there.

Fix $b>a\ge a_0$ and introduce
\begin{equation}\label{eq:transformed-contact-functions}
 U_a(x):=u(x)+\frac a2\abs{x}^2,
 \qquad
 w_a(x):=\Gamma_u^a(x)+\frac a2\abs{x}^2
          =\Gamma_{U_a}^0(x).
\end{equation}
By \eqref{eq:envelope-transform},
\begin{equation}\label{eq:transformed-contact-sets}
 A_0(U_a)=A_a(u),
 \qquad
 A_{b-a}(U_a)=A_b(u).
\end{equation}
Thus the passage from opening $a$ to opening $b$ is represented, in the transformed variables, by paraboloids of opening $-(b-a)$.  This normalization is important: it is the difference $b-a$, rather than $b$, that enters the vertex map.

\begin{proposition}[Exact spectral envelope]
\label{prop:spectral-envelope}
Let $E$ be a set of new contacts obtained by taking paraboloids of
opening $-(b-a)$ tangent from below to $w_a$ and sliding them upward
until they touch $U_a$.  At every $x\in E$ at which $u$ is twice
differentiable in the sense of Alexandrov and
$1\le r(x)\le n-1$,
\begin{equation}\label{eq:spectral-Jacobian}
 \det\left(\frac{bI+D^2u(x)}{b-a}\right)
 \le
 \frac{(b-s(x))^{r(x)}
 \left(b+\dfrac{\kappa r(x)}{n-r(x)}s(x)\right)^{n-r(x)}}{(b-a)^n}.
\end{equation}
For $r(x)=n$, the right-hand side is interpreted as
$(b-s(x))^n/(b-a)^n$; for $r(x)=0$, it is $(b/(b-a))^n$.
\end{proposition}

\begin{proof}
At a new contact point, $D^2U_a\ge-(b-a)I$, and therefore
\begin{equation}\label{eq:Hessian-lower-at-contact}
 D^2u\ge-bI.
\end{equation}
In particular, $0\le q_i\le b$.  The Pucci inequality gives
\begin{equation}\label{eq:Pucci-balance}
 \sum_{j=1}^{n-r}\mu_j
 \le\kappa\sum_{i=1}^rq_i
 =\kappa r s.
\end{equation}
By \eqref{eq:Hessian-lower-at-contact}, the matrix
$bI+D^2u(x)$ is positive semidefinite.  For
$1\le r=r(x)\le n-1$, its eigenvalues are
\[
 b-q_1,\ldots,b-q_r,
 \qquad
 b+\mu_1,\ldots,b+\mu_{n-r}.
\] 
Applying the arithmetic--geometric mean inequality separately to the negative and nonnegative spectral groups, we obtain
\[
 \prod_{i=1}^r(b-q_i)\le(b-s)^r
\]
and, by \eqref{eq:Pucci-balance},
\[
 \prod_{j=1}^{n-r}(b+\mu_j)
 \le\left(b+\frac{\kappa r}{n-r}s\right)^{n-r}.
\]
Multiplication proves \eqref{eq:spectral-Jacobian}; the endpoint cases are immediate.
\end{proof}
The estimate is sharp under the pointwise constraints.  Indeed, for
$1\le r\le n-1$ and $0<s\le b$, every matrix
\[
 H_{r,s,R}
 :=
 R^T\operatorname{diag}\left(
 -sI_r,\frac{\kappa rs}{n-r}I_{n-r}
 \right)R,
 \qquad R\in O(n),
\]
satisfies $H_{r,s,R}\ge-bI$ and
$\Mminus(H_{r,s,R})=0$, and attains equality in
\eqref{eq:spectral-Jacobian}.  Thus
\Cref{prop:spectral-envelope} is the exact pointwise envelope allowed
by the contact constraint and the Pucci cone.

For $x\in\Omega$, define
\begin{equation}\label{eq:tau-definition}
 \tau(x):=\inf\{a\ge0:x\in A_a(u,\Omega)\}.
\end{equation}
The infimum is attained whenever it is finite.  Indeed, if
$a_j\downarrow\tau(x)$ and $P_j$ are supporting paraboloids of opening
$-a_j$ at $x$, their slopes are bounded by evaluating $P_j\le u$ at
$x\pm he$; a subsequence then converges to a supporting paraboloid of
opening $-\tau(x)$.  Consequently,
\begin{equation}\label{eq:contact-set-tau}
 A_t(u,\Omega)\cap\Omega
 =
 \{x\in\Omega:\tau(x)\le t\}.
\end{equation}
For fixed $t$, the envelope $\Gamma_u^t$ is lower semicontinuous, being
a supremum of continuous functions.  Since $u-\Gamma_u^t$ is
nonnegative and upper semicontinuous,
\[
 A_t(u,\Omega)
 =
 \bigcap_{k=1}^{\infty}
 \left\{x\in\overline\Omega:
 u(x)-\Gamma_u^t(x)<\frac1k\right\}
\]
is Borel measurable.  It follows from
\eqref{eq:contact-set-tau} that $\tau$ is Borel measurable on $\Omega$.
Finally, if $u$ is twice differentiable at $x$ and
$\tau(x)<\infty$, the supporting paraboloid of opening $-\tau(x)$ gives
\[
 D^2u(x)\ge-\tau(x)I.
\]
Therefore
\begin{equation}\label{eq:s-le-tau}
 0\le s(x)\le\tau(x).
\end{equation}
For $0\le r\le n$, define the negative-index stratum
\[
 \Sigma_r
 :=
 \left\{
 x\in\Omega:
 \begin{array}{l}
 \text{$u$ is twice differentiable at $x$ in the sense of Alexandrov,}\\
 r(x)=r
 \end{array}
 \right\}.
\]
Each $\Sigma_r$ is measurable.  Locally, semiconcavity makes
$x\mapsto C\abs{x}^2/2-u(x)$ convex for a suitable constant $C$.
Alexandrov's theorem \cite[Theorem~6.9]{EvansGariepy} therefore gives
\[
 \abs{\Omega\setminus\bigcup_{r=0}^n\Sigma_r}=0.
\]
Thus $\Sigma_0,\ldots,\Sigma_n$ form, up to a null set, the measurable
partition of $\Omega$ by the actual negative index of the Hessian.

\begin{proposition}[Shell inequality by negative index]
\label{prop:shell-inequality}
Fix $\rho>1$ and define
\[
 m(a):=\abs{\Omega\setminus A_a(u,\Omega)}.
\]
Then, for every $a\ge a_0$,
\begin{equation}\label{eq:shell-inequality}
 m(a)\le
 \sum_{r=0}^n
 \int_{\{a<\tau\le\rho a\}\cap\Sigma_r}
 \left(\frac{\rho}{\rho-1}\right)^n
 P_{n,r,\kappa}\left(\frac{s(x)}{\rho a}\right)\dd x,
\end{equation}
where
\begin{equation}\label{eq:endpoint-kernels}
 P_{n,0,\kappa}\equiv1,
 \qquad
 P_{n,n,\kappa}(z)=(1-z)^n.
\end{equation}
\end{proposition}

\begin{proof}
By inner regularity, it is enough to prove the estimate with $m(a)$
replaced by $\abs{F}$, where
\[
 F\subset\Omega\setminus A_a(u,\Omega)
   =\{U_a>w_a\}
\]
is compact.  Set
\[
 \delta:=(\rho-1)a>0.
\]
For $x_0\in F$ and $\xi\in\partial w_a(x_0)$, consider the paraboloid
\[
 Q_{x_0,\xi}(y)
 :=
 w_a(x_0)+\xi\cdot(y-x_0)
 -\frac{\delta}{2}\abs{y-x_0}^2.
\]
It touches $w_a$ from below at $x_0$.  Let $\mathcal Q_F$ be the
family of all such paraboloids, and let $V_F$ be their vertex set.
By \Cref{lem:vertex-expansion},
\begin{equation}\label{eq:vertices-dominate-noncontact}
 \abs{V_F}\ge\abs{F}.
\end{equation}

Fix $Q=Q_{x_0,\xi}\in\mathcal Q_F$.  Since
$\xi\in\partial w_a(x_0)$,
\begin{equation}\label{eq:strict-separation-from-envelope}
 w_a(y)-Q(y)
 \ge\frac{\delta}{2}\abs{y-x_0}^2
 \qquad(y\in\overline\Omega).
\end{equation}
At $x_0$ one has
\[
 U_a(x_0)>w_a(x_0)=Q(x_0),
\]
whereas \eqref{eq:strict-separation-from-envelope} and $U_a\ge w_a$
give $U_a(y)>Q(y)$ whenever $y\ne x_0$.  Since $U_a-Q$ is continuous
on $\overline\Omega$, it follows that
\[
 t_Q:=\min_{\overline\Omega}(U_a-Q)>0.
\]
Let $E$ be the union, over $Q\in\mathcal Q_F$, of the corresponding
first-contact sets:
\[
 E
 :=
 \bigcup_{Q\in\mathcal Q_F}
 \{x\in\overline\Omega:U_a(x)=Q(x)+t_Q\}.
\]

We first show that every first contact is interior.  Suppose that
$x_1\in\partial\Omega$ is a first-contact point associated with
$Q\in\mathcal Q_F$, whose point of tangency with $w_a$ is
$x_0\in F$.  Since every boundary point belongs to
$A_0(U_a,\Omega)$,
\[
 U_a(x_1)=w_a(x_1).
\]
Set
\[
 H:=w_a-Q.
\]
The function $H$ is $\delta$-strongly convex, and
\[
 H(x_0)=0,
 \qquad
 H(x_1)=t_Q.
\]
For $0<\theta<1$, put
\[
 x_\theta:=(1-\theta)x_0+\theta x_1.
\]
Since
\[
 \Omega\setminus A_0(U_a,\Omega)
 =
 \Omega\setminus A_a(u,\Omega)
 \Subset\Omega,
\]
one has $x_\theta\in A_0(U_a,\Omega)$ when $\theta$ is sufficiently
close to $1$.  Thus $U_a(x_\theta)=w_a(x_\theta)$, while strong
convexity gives
\begin{align*}
 U_a(x_\theta)-Q(x_\theta)
 &=H(x_\theta)\\
 &\le
 (1-\theta)H(x_0)+\theta H(x_1)
 -\frac{\delta}{2}\theta(1-\theta)\abs{x_1-x_0}^2\\
 &=
 \theta t_Q
 -\frac{\delta}{2}\theta(1-\theta)\abs{x_1-x_0}^2
 <t_Q.
\end{align*}
This contradicts the definition of $t_Q$.  Hence
\[
 E\subset\Omega.
\]

Moreover, in the viscosity sense,
\[
 \Mminus(D^2U_a)
 \le\Mminus(D^2u)+\Mplus(aI)
 \le n\Lambda a
 \qquad\text{in }\Omega.
\]
We may therefore apply \Cref{lem:separation} to every
$Q\in\mathcal Q_F$.  It follows that
\[
 E\cap A_0(U_a,\Omega)=\varnothing.
\]
Since each lifted paraboloid $Q+t_Q$ has opening $-\delta$, every
point of $E$ belongs to $A_\delta(U_a,\Omega)$.  Recalling
\eqref{eq:transformed-contact-sets}, we obtain
\begin{equation}\label{eq:new-contacts-in-shell}
 E\subset
 A_{\rho a}(u,\Omega)\setminus A_a(u,\Omega).
\end{equation}

We shall use the area formula on $E$, so we record its compactness.
Because $F\Subset\Omega$ and $w_a$ is convex, its subgradients are
uniformly bounded on $F$.  Moreover, the graph of $\partial w_a$ is
closed.  Hence
\[
 \mathcal G_F
 :=
 \{(x_0,\xi):x_0\in F,\ \xi\in\partial w_a(x_0)\}
\]
is compact.  The map
\[
 (x_0,\xi)\longmapsto Q_{x_0,\xi}
\]
is continuous in the uniform norm on $\overline\Omega$, and so is
\[
 Q\longmapsto t_Q
 =\min_{\overline\Omega}(U_a-Q).
\]
Consequently, the incidence set
\[
 \left\{
 (Q,x)\in\mathcal Q_F\times\overline\Omega:
 U_a(x)=Q(x)+t_Q
 \right\}
\]
is compact.  Its projection onto the second factor is $E$, and hence
$E$ is compact.  Since $E\cap\partial\Omega=\varnothing$, we conclude
that
\[
 E\Subset\Omega.
\]

We next identify the vertex map and its differential.  For every
$x\in E$, choose a first-contact paraboloid and write it as
\[
 P_x(y)
 =
 U_a(x)+\zeta_x\cdot(y-x)
 -\frac{\delta}{2}\abs{y-x}^2,
 \qquad
 P_x\le U_a\quad\text{in }\overline\Omega.
\]
Define
\[
 g(y)
 :=
 U_a(y)+\frac{\delta}{2}\abs{y}^2
 =
 u(y)+\frac{\rho a}{2}\abs{y}^2.
\]
Adding $\delta\abs{y}^2/2$ to the supporting inequality gives
\[
 g(x)+p_x\cdot(y-x)\le g(y)
 \qquad(y\in\overline\Omega),
 \qquad
 p_x:=\zeta_x+\delta x.
\]

The function $g$ is semiconcave in a neighborhood of $E$.  For small
$h$, set
\[
 R_x(h):=g(x+h)-g(x)-p_x\cdot h.
\]
The affine support gives $R_x(h)\ge0$ and $R_x(-h)\ge0$, while
semiconcavity gives
\[
 R_x(h)+R_x(-h)\le C\abs{h}^2.
\]
It follows that
\[
 0\le R_x(h)\le C\abs{h}^2.
\]
Thus $g$ is differentiable at every $x\in E$, with
\[
 p_x=\nabla g(x).
\]
In particular, $p_x$ is independent of the chosen first-contact
paraboloid.

The map $x\mapsto p_x$ is Lipschitz on $E$.  Indeed, let $C$ be a
semiconcavity constant for $g$ in a neighborhood of $E$.  Given
$x,y\in E$, put $d:=\abs{x-y}$.  If $d$ is sufficiently small and
$p_x\ne p_y$, set
\[
 e:=\frac{p_y-p_x}{\abs{p_y-p_x}},
 \qquad
 z:=x+de.
\]
The affine support at $y$ and the semiconcavity inequality at $x$
give
\[
 g(y)+p_y\cdot(z-y)
 \le g(z)
 \le g(x)+p_x\cdot(z-x)+\frac C2d^2.
\]
Semiconcavity at $y$, evaluated at $x$, also gives
\[
 g(y)-g(x)+p_y\cdot(x-y)\ge-\frac C2d^2.
\]
Combining these inequalities yields
\[
 d\abs{p_y-p_x}\le Cd^2.
\]
Hence
\[
 \abs{p_y-p_x}\le C\abs{x-y}
\]
for nearby points of $E$.

Since $E\Subset\Omega$, the affine supporting inequalities may be
evaluated at $x\pm he$, with one fixed $h>0$, uniformly for
$x\in E$ and $e\in\mathbb S^{n-1}$.  This shows that the slopes
$p_x$ are uniformly bounded.  After increasing the constant, the
same Lipschitz estimate therefore holds for every $x,y\in E$.

The vertex of $P_x$ is given by the well-defined map
\[
 \Phi(x)
 :=
 \frac{p_x}{\delta}
 =
 x+\frac{\nabla U_a(x)}{\delta}
 =
 \frac{\rho a\,x+\nabla u(x)}{(\rho-1)a}.
\]
Thus $\Phi:E\to\Rn$ is Lipschitz.  Vertical translation does not
change the vertex of a paraboloid: if $Q\in\mathcal Q_F$ and
$Q+t_Q$ first touches $U_a$ at $x$, then
\[
 \Phi(x)=\operatorname{vertex}(Q).
\]
Consequently,
\begin{equation}\label{eq:vertices-contained-in-image}
 V_F\subset\Phi(E).
\end{equation}

For clarity, if $x$ is a density point of $E$, we write
$\operatorname{ap}D\Phi(x)=L$ when, for every $\varepsilon>0$,
\[
 \lim_{r\downarrow0}
 \frac{
 \abs{
 \left\{
 y\in E\cap B_r(x):
 \abs{\Phi(y)-\Phi(x)-L(y-x)}
 >
 \varepsilon\abs{y-x}
 \right\}}
 }{\abs{B_r}}
 =0.
\]
Let $x\in E$ be a point at which $g$ is twice differentiable in the
sense of Alexandrov.  Set $H:=D^2g(x)$ and write
\[
 g(x+k)
 =
 g(x)+p_x\cdot k+\frac12k\cdot Hk+R(k),
 \qquad
 R(k)=o(\abs{k}^2).
\]
Let $y=x+h\in E$, fix $0<t<1$, and put
\[
 k:=t\abs{h}e,
 \qquad
 e\in\mathbb S^{n-1}.
\]
For $y$ sufficiently close to $x$, the affine support at $y$ may be
evaluated at $y\pm k$, and therefore
\[
 \frac{g(y)-g(y-k)}{t\abs{h}}
 \le p_y\cdot e
 \le
 \frac{g(y+k)-g(y)}{t\abs{h}}.
\]
Expanding the three values of $g$ at $x$ gives
\begin{align*}
 \frac{g(y+k)-g(y)}{t\abs{h}}
 &=
 p_x\cdot e+Hh\cdot e
 +\frac{t\abs{h}}2e\cdot He
 +\frac{R(h+k)-R(h)}{t\abs{h}},\\
 \frac{g(y)-g(y-k)}{t\abs{h}}
 &=
 p_x\cdot e+Hh\cdot e
 -\frac{t\abs{h}}2e\cdot He
 +\frac{R(h)-R(h-k)}{t\abs{h}}.
\end{align*}
For fixed $t$, the remainder estimates are uniform in
$e\in\mathbb S^{n-1}$.  Taking the supremum over $e$ yields
\[
 \abs{p_y-p_x-Hh}
 \le
 \frac t2\abs{H}\abs{h}
 +\eta_t(\abs{h})\abs{h},
\]
where $\eta_t(r)\to0$ as $r\downarrow0$.  Dividing by $\abs{h}$,
first letting $y\to x$ through $E$, and then letting $t\downarrow0$,
we obtain
\[
 \lim_{\substack{E\ni y\to x\\y\ne x}}
 \frac{
 \abs{p_y-p_x-D^2g(x)(y-x)}
 }{\abs{y-x}}
 =0.
\]
Thus, whenever $x$ is also a density point of $E$,
\begin{equation}\label{eq:approximate-differential-vertex-map}
 \operatorname{ap}D\Phi(x)
 =
 \frac{D^2g(x)}{\delta}
 =
 \frac{\rho aI+D^2u(x)}{(\rho-1)a}.
\end{equation}
By Alexandrov's theorem and the Lebesgue density theorem, this identity
holds for almost every $x\in E$; see
\cite[Theorems~1.35 and~6.9]{EvansGariepy}.

Let $\widetilde\Phi:\Rn\to\Rn$ be a Lipschitz extension of $\Phi$.
At almost every $x\in E$, the point $x$ is a density point of $E$,
the map $\widetilde\Phi$ is differentiable at $x$, and
\eqref{eq:approximate-differential-vertex-map} holds.  Since
$\widetilde\Phi=\Phi$ on $E$, uniqueness of the approximate
differential gives
\[
 D\widetilde\Phi(x)=\operatorname{ap}D\Phi(x)
 \qquad\text{for almost every }x\in E.
\]
The area formula applied to $\widetilde\Phi$ on $E$ therefore yields
\begin{align*}
 \abs{\Phi(E)}
 &\le
 \int_{\Rn}N(\widetilde\Phi,E,z)\,\dd z\\
 &=
 \int_E\abs{\det D\widetilde\Phi(x)}\,\dd x\\
 &=
 \int_E
 \abs{\det\bigl(\operatorname{ap}D\Phi(x)\bigr)}\,\dd x,
\end{align*}
where $N(\widetilde\Phi,E,z)$ denotes the multiplicity of
$\widetilde\Phi$ on $E$ at $z$.  Here we used the Lipschitz-extension,
Rademacher, and area-formula theorems; see
\cite[Theorems~3.1, 3.2, and~3.8]{EvansGariepy}.

At almost every $x\in E$, the function $U_a-P_x$ has a local minimum
at $x$, and hence
\[
 \rho aI+D^2u(x)
 =
 D^2U_a(x)+\delta I
 \ge0.
\]
The determinant in
\eqref{eq:approximate-differential-vertex-map} is therefore
nonnegative.  Since $\Sigma_0,\ldots,\Sigma_n$ partition $E$ up to a
null set, \eqref{eq:vertices-dominate-noncontact},
\eqref{eq:vertices-contained-in-image}, the area formula, and
\Cref{prop:spectral-envelope}, with $b=\rho a$, give
\begin{align*}
 \abs{F}
 &\le\abs{V_F}
 \le\abs{\Phi(E)}\\
 &\le
 \int_E
 \abs{\det\bigl(\operatorname{ap}D\Phi(x)\bigr)}
 \dd x\\
 &=
 \sum_{r=0}^n
 \int_{E\cap\Sigma_r}
 \det\left(
 \frac{\rho aI+D^2u(x)}{(\rho-1)a}
 \right)\dd x\\
 &\le
 \sum_{r=0}^n
 \int_{E\cap\Sigma_r}
 \left(\frac{\rho}{\rho-1}\right)^n
 P_{n,r,\kappa}
 \left(\frac{s(x)}{\rho a}\right)\dd x.
\end{align*}

Finally, \eqref{eq:contact-set-tau} and
\eqref{eq:new-contacts-in-shell} imply
\[
 E\subset\{a<\tau\le\rho a\}.
\]
For every $x$ in this shell,
\[
 0\le\frac{s(x)}{\rho a}
 \le\frac{\tau(x)}{\rho a}
 \le1.
\]
All the mixed and endpoint kernels are nonnegative on $[0,1]$.
Therefore each set $E\cap\Sigma_r$ may be enlarged to
$\{a<\tau\le\rho a\}\cap\Sigma_r$, and we obtain
\[
 \abs{F}
 \le
 \sum_{r=0}^n
 \int_{\{a<\tau\le\rho a\}\cap\Sigma_r}
 \left(\frac{\rho}{\rho-1}\right)^n
 P_{n,r,\kappa}
 \left(\frac{s(x)}{\rho a}\right)\dd x.
\]
Taking the supremum over compact
$F\subset\Omega\setminus A_a(u,\Omega)$ proves
\eqref{eq:shell-inequality}.
\end{proof}

\begin{remark}[Why the pointwise negative index matters]\label{rem:actual-index}
The integer $r$ in \Cref{prop:spectral-envelope} is the actual number of negative eigenvalues at the contact point, and it may vary across the contact set.  It cannot be replaced inside the determinant estimate by a global lower bound.  The decomposition in \eqref{eq:shell-inequality} is precisely what allows all possible values of $r$ to be treated simultaneously without sacrificing the pointwise information.
\end{remark}

\section{Continuum integration in the opening}\label{sec:continuum}

We now turn \eqref{eq:shell-inequality} into an integrability estimate.  The only optimization that remains is one-dimensional.

\subsection{Critical exponents for the negative index}

For $1\le r\le n-1$, differentiation of \eqref{eq:index-kernel-intro} gives
\begin{equation}\label{eq:P-prime-factorization}
\begin{split}
 P'_{n,r,\kappa}(z)
 ={}&r(1-z)^{r-1}
 \left(1+\frac{\kappa r}{n-r}z\right)^{n-r-1}\\
 &\times\left[(\kappa-1)-\frac{\kappa n}{n-r}z\right].
\end{split}
\end{equation}
For $\kappa>1$, the derivative changes sign exactly once, from positive to negative, at
\begin{equation}\label{eq:P-maximizer}
 z_{n,r,\kappa}^*=\frac{(\kappa-1)(n-r)}{\kappa n}.
\end{equation}

\begin{lemma}[Existence and uniqueness of the critical exponent]\label{lem:critical-exponent}
Let $\kappa>1$ and $1\le r\le n-1$.  The function
\[
 \mathcal I_{n,r,\kappa}(p)
 :=\int_0^1z^{-p}P'_{n,r,\kappa}(z)\dd z,
 \qquad 0\le p<1,
\]
has a unique zero in $(0,1)$.  Moreover,
\begin{equation}\label{eq:I-sign}
 \mathcal I_{n,r,\kappa}(p)<0
 \quad\Longleftrightarrow\quad
 p<p_{n,r}(\kappa).
\end{equation}
\end{lemma}

\begin{proof}
Suppress the indices and write
\[
 P:=P_{n,r,\kappa},
 \qquad
 \mathcal I:=\mathcal I_{n,r,\kappa},
 \qquad
 z_*:=z_{n,r,\kappa}^*.
\]
Since $P(0)=1$ and $P(1)=0$,
\[
 \mathcal I(0)=\int_0^1P'(z)\dd z=-1.
\]
Moreover, $P'(0)=r(\kappa-1)>0$, and therefore
\[
 \mathcal I(p)\longrightarrow+\infty
 \qquad\text{as }p\uparrow1.
\]
Continuity gives at least one zero in $(0,1)$.

Suppose that $\mathcal I(p_0)=0$ and let $p>p_0$.  Set
\[
 f(z):=z^{-p_0}P'(z).
\]
Then
\[
 \int_0^1f(z)\dd z=0,
\]
while \eqref{eq:P-prime-factorization} shows that $f$ is positive on
$(0,z_*)$ and negative on $(z_*,1)$.  Since
$z\mapsto z^{-(p-p_0)}$ is strictly decreasing, subtraction of its
value at $z_*$ gives
\begin{align*}
 \mathcal I(p)
 &=
 \int_0^1 z^{-(p-p_0)}f(z)\dd z\\
 &=
 \int_0^1
 \left(
 z^{-(p-p_0)}-z_*^{-(p-p_0)}
 \right)f(z)\dd z
 >0.
\end{align*}
Indeed, the integrand is positive on both $(0,z_*)$ and $(z_*,1)$.
Thus $\mathcal I$ can have no second zero to the right of $p_0$.
Applying the same conclusion to any hypothetical zero to the left of
$p_0$ proves uniqueness.  Since $\mathcal I(0)<0$, continuity then
yields
\[
 \mathcal I(p)<0
 \quad\Longleftrightarrow\quad
 p<p_{n,r}(\kappa).
\]
\end{proof}

For $0<p<1$, define the aperture potential
\begin{equation}\label{eq:Psi-definition}
 \Psi_{n,r,\kappa,p}(y)
 :=p y^p\int_0^y z^{-p-1}
 \big(P_{n,r,\kappa}(z)-1\big)\dd z,
 \qquad 0\le y\le1.
\end{equation}
For $r=0$ and $r=n$, the same notation refers to the endpoint kernels in \eqref{eq:endpoint-kernels}.

\begin{lemma}[The aperture potential]\label{lem:Psi-contraction}
If $1\le r\le n-1$ and $0<p<p_{n,r}(\kappa)$, then
\begin{equation}\label{eq:Psi-strict}
 \sup_{0\le y\le1}\Psi_{n,r,\kappa,p}(y)<1.
\end{equation}
For $r=0$ or $r=n$, the same conclusion holds for every $0<p<1$.
\end{lemma}

\begin{proof}
Suppress the indices.  Since $P(z)-1=O(z)$ at the origin, $\Psi$ is continuous on $[0,1]$ and $\Psi(0)=0$.  Direct differentiation gives
\begin{equation}\label{eq:Psi-derivative}
 y\Psi'(y)=p\big(\Psi(y)+P(y)-1\big).
\end{equation}
At an interior maximum,
\[
 \Psi(y)=1-P(y)<1.
\]
At $y=1$, integration by parts yields
\begin{equation}\label{eq:Psi-at-one}
 \Psi(1)=1+\int_0^1z^{-p}P'(z)\dd z<1.
\end{equation}
For $\kappa>1$, the last inequality follows from \Cref{lem:critical-exponent}; for $\kappa=1$, \eqref{eq:P-prime-factorization} gives $P'\le0$ on $(0,1)$.  Compactness of $[0,1]$ proves \eqref{eq:Psi-strict}.  The case $r=0$ is immediate from $P\equiv1$, while for $r=n$ one again has $P'\le0$.
\end{proof}

\subsection{Full and truncated aperture intervals}

Fix $\rho>1$.  At a point with $\tau(x)=T$ and $s(x)=yT$, the point belongs to the shell $A_{\rho a}\setminus A_a$ precisely when $T/\rho\le a<T$.  Its normalized contribution over this full aperture interval is
\begin{equation}\label{eq:full-aperture-coefficient}
 \mathcal C_{n,r,\kappa,p,\rho}(y)
 :=p\left(\frac{\rho}{\rho-1}\right)^n\rho^{-p}
 \int_1^\rho t^{p-1}P_{n,r,\kappa}\left(\frac yt\right)\dd t.
\end{equation}

\begin{lemma}[Large-ratio contraction]\label{lem:large-rho-contraction}
If $0<p<\underline p_n(\kappa)$, there are $\rho>1$ and $0<\theta<1$, depending only on $n$, $\kappa$, and $p$, such that
\begin{equation}\label{eq:uniform-full-contraction}
 \mathcal C_{n,r,\kappa,p,\rho}(y)\le\theta
 \qquad(0\le r\le n,\ 0\le y\le1).
\end{equation}
\end{lemma}

\begin{proof}
Again suppress the indices.  Split the integral in \eqref{eq:full-aperture-coefficient} into the contribution of $1$ and that of $P-1$.  The change of variables $z=y/t$ gives, uniformly for $0\le y\le1$,
\begin{equation}\label{eq:C-asymptotic}
 \mathcal C_{n,r,\kappa,p,\rho}(y)
 =1+\big(\Psi_{n,r,\kappa,p}(y)-1\big)\rho^{-p}
 +O(\rho^{-1}).
\end{equation}
Indeed, the part omitted from the $z$-integral lies in $(0,y/\rho)$ and is $O(\rho^{-1})$, because $P(z)-1=O(z)$.  Since $p<1$, the error in \eqref{eq:C-asymptotic} is $o(\rho^{-p})$.  By \Cref{lem:Psi-contraction} and the finiteness of the index family, there is $\delta>0$ such that
\[
 \Psi_{n,r,\kappa,p}(y)\le1-\delta
\]
for every $r$ and $y$.  Choosing $\rho$ sufficiently large gives \eqref{eq:uniform-full-contraction}.
\end{proof}

The last interval in a truncated moment integral is only a portion of the full aperture range.  The following monotonicity shows that this cannot increase the coefficient.

\begin{lemma}[Terminal aperture cutoff]\label{lem:upper-cutoff}
Assume $1\le r\le n-1$ and $0<p<p_{n,r}(\kappa)$.  For $0<\ell\le1$, $0\le y\le1$, and $P=P_{n,r,\kappa}$, set
\begin{equation}\label{eq:H-definition}
 H_y(\ell):=\int_\ell^1t^{p-1}P\left(\frac{y\ell}{t}\right)\dd t.
\end{equation}
Then $H_y$ is nonincreasing.  The same conclusion holds for $r=0,n$ and every $0<p<1$.
\end{lemma}

\begin{proof}
The case $y=0$ is immediate.  Assume $y>0$.  Differentiating and changing variables $z=y\ell/t$, we find
\begin{equation}\label{eq:H-derivative}
 H_y'(\ell)=\ell^{p-1}
 \left[-P(y)+y^p\int_{y\ell}^{y}z^{-p}P'(z)\dd z\right].
\end{equation}
For fixed $y$, write
\[
 J(a):=\int_a^y z^{-p}P'(z)\dd z,
 \qquad 0\le a\le y.
\]
If $P'\le0$ on $(0,y)$, then $J(a)\le0$, and
\eqref{eq:H-derivative} is nonpositive.  Otherwise,
\eqref{eq:P-prime-factorization} shows that $P'$ has at most one sign
change on $[0,y]$, from positive to negative.  Hence
\[
 J'(a)=-a^{-p}P'(a)
\]
has at most one sign change, from negative to positive.  Thus $J$ has
no strict interior maximum, and its maximum on $[0,y]$ is attained at
an endpoint. Now $J(y)=0$.  Since $P(z)-1=O(z)$ and $p<1$, the boundary term
$z^{-p}(P(z)-1)$ vanishes as $z\downarrow0$.  Integration by parts
and \eqref{eq:Psi-definition} therefore give
\begin{equation}\label{eq:J-zero}
 y^pJ(0)=P(y)-1+\Psi(y)<P(y).
\end{equation}
Therefore
\[
 y^pJ(y\ell)\le\max\{y^pJ(0),0\}\le P(y),
\]
and \eqref{eq:H-derivative} is nonpositive.  For $r=0$, the assertion follows immediately from $P\equiv1$; for
$r=n$, one has $P(z)=(1-z)^n$ and $P'\le0$, so the first case applies.
\end{proof}

\subsection{The moment estimate}

\begin{proof}[Proof of \Cref{thm:general-lower-intro}]
We first treat a semiconcave supersolution.  Apply \Cref{lem:localization} and denote the localized function again by $u$.  It remains semiconcave: after increasing the semiconcavity constant if necessary, the two functions entering the minimum have a common constant, and their minimum is semiconcave.  Write $\Omega=B_R$, and let $a_0$ be as in \Cref{lem:localization}.  Fix
\[
 0<p<\underline p_n(\kappa),
\]
and choose $\rho$ and $\theta$ from \Cref{lem:large-rho-contraction}.

For $L\ge\rho a_0$, put
\begin{equation}\label{eq:ML-definition}
 M_L:=p\int_{a_0}^L a^{p-1}m(a)\dd a,
 \qquad m(a)=\abs{\{\tau>a\}}.
\end{equation}
Multiply \eqref{eq:shell-inequality} by $pa^{p-1}$, integrate over $a\in[a_0,L]$, and use Tonelli's theorem.  We organize the resulting integral according to $T=\tau(x)$.

If $a_0<T<\rho a_0$, the aperture interval is truncated at its lower end.  Since the finitely many kernels are bounded on $[0,1]$, the contribution of each such point is at most $Ca_0^p$.  After integration in $x$, this range contributes at most $Ca_0^p\abs{\Omega}$.

If $\rho a_0\le T\le L$, the complete interval $[T/\rho,T]$ occurs.  By \eqref{eq:uniform-full-contraction}, the contribution is at most $\theta T^p$.

Finally, let $L<T\le\rho L$ and write $T=\rho\ell L$, where $\ell\in[1/\rho,1]$.  With $y=s(x)/T$ and $a=Lt$, the contribution divided by $L^p$ is
\[
 p\left(\frac\rho{\rho-1}\right)^nH_y(\ell).
\]
By \Cref{lem:upper-cutoff}, this is no larger than its value at $\ell=1/\rho$, which is precisely the full coefficient in \eqref{eq:full-aperture-coefficient}.  It is therefore at most $\theta$.  Points with $T>\rho L$ occur in none of the shells under consideration.

Enlarging $\{L<T\le\rho L\}$ to $\{T>L\}$, we have proved
\begin{equation}\label{eq:ML-inequality}
 M_L\le Ca_0^p\abs{\Omega}
 +\theta\left(
 \int_{\{a_0<\tau\le L\}}\tau^p\dd x
 +L^p\abs{\{\tau>L\}}\right).
\end{equation}
The layer-cake identity gives
\begin{equation}\label{eq:layer-identity}
 \int_{\{a_0<\tau\le L\}}\tau^p\dd x
 +L^p\abs{\{\tau>L\}}
 =M_L+a_0^pm(a_0).
\end{equation}
Combining \eqref{eq:ML-inequality} and \eqref{eq:layer-identity}, and using $m(a_0)\le\abs{\Omega}$, we obtain
\[
 (1-\theta)M_L\le Ca_0^p\abs{\Omega}
\]
uniformly in $L$.  Letting $L\to\infty$ in the preceding identity and using monotone
convergence, we obtain
\[
 \int_{\{\tau>a_0\}}\tau^p\,\dd x
 \le
 \lim_{L\to\infty}M_L+a_0^p m(a_0)
 \le C.
\]
The complementary set contributes at most $a_0^p\abs{\Omega}$, and hence
\begin{equation}\label{eq:localized-moment}
 \int_\Omega\tau(x)^p\dd x\le C(n,\kappa,p).
\end{equation}
Here we have used that $a_0$ and $\Omega$ were fixed in \Cref{lem:localization}.

Return to the normalization in that lemma.  If
\[
 v=1+\frac{\eta}{2\norm{u}_\infty}
 (u+\norm{u}_\infty),
\]
then
\[
 \ThetaU(v,B_1)=\frac{\eta}{2\norm{u}_\infty}\ThetaU(u,B_1).
\]
The last assertion of \Cref{lem:localization} and \eqref{eq:localized-moment} prove \eqref{eq:main-general-estimate-intro} for semiconcave supersolutions.

For a general continuous supersolution, take the inf-convolutions in \Cref{lem:inf-convolution}.  Applying the semiconcave estimate to $u_j$ in $B_{R_j}$ and scaling to the unit ball gives
\begin{equation}\label{eq:inf-convolution-estimate}
 \int_{B_{R_j/2}}\ThetaU(u_j,B_{R_j})^p\dd x
 \le C R_j^{n-2p}\norm{u}_{L^\infty(B_1)}^p
 \le C\norm{u}_{L^\infty(B_1)}^p,
\end{equation}
where the last inequality uses $p<1$ and $n\ge2$.  Since $R_j/2\uparrow1/2$, \eqref{eq:Theta-lower-semicontinuity} and Fatou's lemma yield \eqref{eq:main-general-estimate-intro}.  The case $u\equiv0$ is immediate.
\end{proof}

\begin{remark}[Tail estimate]\label{rem:tail-estimate}
By Chebyshev's inequality, for every $p<\underline p_n(\kappa)$,
\[
 \abs{\{x\in B_{1/2}:\ThetaU(u,B_1)(x)>t\norm{u}_\infty\}}
 \le C t^{-p},
 \qquad t>0.
\]
At points of twice differentiability, \eqref{eq:Hessian-controlled-by-Theta} follows directly from the same spectral balance used in \eqref{eq:Pucci-balance}.
\end{remark}

\section{The plane and dimension three}\label{sec:low-dimensions}

\subsection{The plane}

For $n=2$, there is only one mixed negative index, and
\begin{equation}\label{eq:planar-P}
 P_{2,1,\kappa}(z)=(1-z)(1+\kappa z).
\end{equation}
Consequently,
\begin{equation}\label{eq:planar-critical-calculation}
 \int_0^1z^{-p}P'_{2,1,\kappa}(z)\dd z
 =\frac{\kappa-1}{1-p}-\frac{2\kappa}{2-p}.
\end{equation}
For $\kappa>1$, the unique zero is
\begin{equation}\label{eq:planar-critical-root}
 p_{2,1}(\kappa)=\frac{2}{\kappa+1}.
\end{equation}
When $\kappa=1$, the integral in \eqref{eq:planar-critical-calculation} is negative for every $p<1$, and the same formula gives the endpoint value fixed in our convention.  The lower bound in \Cref{thm:planar-intro} follows from \Cref{thm:general-lower-intro}.  The reverse inequality is the planar construction of \cite[Remark~3.3]{ArmstrongSilvestreSmart}; the case $\kappa =1$ is elementary.  This proves \Cref{thm:planar-intro}.

The sharp power is also encoded by an exact identity.  If
\[
 p_*:=\frac{2}{\kappa+1},
\]
then, for $0<t<b$,
\begin{equation}\label{eq:planar-kernel-identity}
 -\frac{\dd}{\dd t}\left[(b-t)^2t^{-p_*}\right]
 =p_*(b-t)(b+\kappa t)t^{-p_*-1}.
\end{equation}

This identity exhibits $p_*$ as the exact exponent selected by the
unsupremized planar contact kernel.  The loss in the scalar discrete
iteration arises from maximizing $(b-t)(b+\kappa t)$ before integrating
in $t$.

\subsection{Dimension three}

There are now two mixed spectral configurations.  For two negative eigenvalues,
\begin{equation}\label{eq:P32}
 P_{3,2,\kappa}(z)=(1-z)^2(1+2\kappa z).
\end{equation}
A direct beta-function calculation gives
\begin{equation}\label{eq:p32}
 p_{3,2}(\kappa)=\frac{3}{2\kappa+1}.
\end{equation}
For one negative eigenvalue,
\begin{equation}\label{eq:P31}
 P_{3,1,\kappa}(z)=(1-z)\left(1+\frac\kappa2z\right)^2.
\end{equation}
Writing the critical integral through the coefficients of this cubic shows that its zero in $(0,1)$ satisfies
\begin{equation}\label{eq:p31-quadratic}
 (\kappa+2)^2p^2-(\kappa+2)(\kappa+10)p+24=0.
\end{equation}
Therefore
\begin{equation}\label{eq:p31}
 p_{3,1}(\kappa)
 =\frac{\kappa+10-\sqrt{\kappa^2+20\kappa+4}}{2(\kappa+2)}.
\end{equation}
The formula extends continuously to $p_{3,1}(1)=1$.  The second root of \eqref{eq:p31-quadratic} is larger than one, since the quadratic is negative at $p=1$ whenever $\kappa>1$.  Substituting $3/(2\kappa+1)$ into its left-hand side gives
\begin{equation}\label{eq:3d-comparison-factor}
 -\frac{6\kappa(\kappa-1)(\kappa-4)}{(2\kappa+1)^2}.
\end{equation}
It follows that
\begin{equation}\label{eq:3d-index-comparison}
 p_{3,2}(\kappa)\le p_{3,1}(\kappa)
 \quad\Longleftrightarrow\quad
 1\le\kappa\le4.
\end{equation}
This proves the lower estimate \eqref{eq:3d-lower-intro}.  The upper bound furnished by the packed radial-profile construction \eqref{eq:radial-upper-intro}, proved in \cite[Theorem~1, equation~(1.5)]{NascimentoTeixeiraJMPA}, then yields \eqref{eq:3d-exact-intro} and \eqref{eq:3d-gap-intro}.  This completes the proof of \Cref{thm:three-dimensional-intro}.

\begin{remark}[The transition at $\kappa=4$]\label{rem:kappa-four}
For a symmetric matrix $M\in\mathcal S^3$ with eigenvalues
$\nu_1,\nu_2,\nu_3$, write
\[
 \sigma_2(M):=
 \nu_1\nu_2+\nu_1\nu_3+\nu_2\nu_3.
\]
Suppose that, at a point with one negative direction, $D^2u$ has
eigenvalues
\[
 -q,\mu_1,\mu_2,
 \qquad
 q>0,\quad \mu_1,\mu_2\ge0.
\]
The Pucci inequality gives
\[
 \mu_1+\mu_2\le\kappa q.
\]
Consequently,
\begin{align*}
 \sigma_2(D^2u)
 &=\mu_1\mu_2-q(\mu_1+\mu_2)\\
 &\le
 \frac14(\mu_1+\mu_2)^2-q(\mu_1+\mu_2).
\end{align*}
If $\kappa\le4$, then $\mu_1+\mu_2\le4q$, and the right-hand side is
nonpositive.

For $\kappa>4$, this conclusion fails sharply.  Indeed, the matrix with
eigenvalues
\[
 -q,\frac{\kappa q}{2},\frac{\kappa q}{2}
\]
lies on the boundary of the Pucci cone, since
\[
 \Mminus\left(
 \operatorname{diag}
 \left(-q,\frac{\kappa q}{2},\frac{\kappa q}{2}\right)
 \right)
 =\lambda\kappa q-\Lambda q=0,
\]
whereas
\[
 \sigma_2
 =
 \frac{\kappa^2q^2}{4}-\kappa q^2
 =
 \frac{\kappa(\kappa-4)}{4}q^2>0.
\]
Thus $\kappa=4$ is exactly the value at which the one-negative-direction
portion of the Pucci cone ceases to be contained in
$\{\sigma_2\le0\}$.  The same transition is detected, without
approximation, by \eqref{eq:3d-index-comparison}.
\end{remark}

\section{Higher-dimensional consequences}\label{sec:higher-dimensions}

The sharp planar threshold and the exact three-dimensional regime are
the main conclusions of this paper.  We record here three consequences
of the general kernels that place the method in the context of the
earlier higher-dimensional theory.

\begin{proposition}\label{prop:last-index}
For every $n\ge2$ and $\kappa\ge1$,
\begin{equation}\label{eq:last-index-exact}
 p_{n,n-1}(\kappa)=\frac{n}{(n-1)\kappa+1}.
\end{equation}
\end{proposition}

\begin{proof}
The assertion at $\kappa=1$ follows from the convention $p_{n,r}(1)=1$.  Assume $\kappa>1$ and put $a=(n-1)\kappa$.  Then
\[
 P_{n,n-1,\kappa}(z)=(1-z)^{n-1}(1+az)
\]
and
\[
 P'_{n,n-1,\kappa}(z)
 =(1-z)^{n-2}\big(a-(n-1)-an z\big).
\]
We write
\[
 B(\alpha,\beta)
 :=
 \int_0^1t^{\alpha-1}(1-t)^{\beta-1}\dd t,
 \qquad \alpha,\beta>0,
\]
for Euler's beta function. We have
\begin{align*}
 \int_0^1z^{-p}P'_{n,n-1,\kappa}(z)\dd z
 ={}&\big(a-(n-1)\big)B(1-p,n-1)\\
 &-anB(2-p,n-1).
\end{align*}
Using
\[
 B(2-p,n-1)=\frac{1-p}{n-p}B(1-p,n-1),
\]
one sees that the integral vanishes exactly at $p=n/(a+1)$.
\end{proof}

Thus the branch with $n-1$ negative eigenvalues always reproduces the radial-profile upper scale.
Its position within the kernel family follows from a comparison with
the neighboring index.

\begin{proposition}[Comparison with the radial branch]
\label{prop:radial-kernel-threshold}
For every $n\ge2$,
\begin{equation}\label{eq:radial-kernel-threshold}
 \left\{\kappa\ge1:
 \underline p_n(\kappa)=p_{n,n-1}(\kappa)\right\}
 =
 \begin{cases}
 [1,\infty),&n=2,\\
 [1,4],&n=3,\\
 \{1\},&n\ge4.
 \end{cases}
\end{equation}
Moreover, if $n\ge4$ and $\kappa>1$, then
\begin{equation}\label{eq:adjacent-branch-below-radial}
 p_{n,n-2}(\kappa)<p_{n,n-1}(\kappa).
\end{equation}
\end{proposition}

\begin{proof}
Assume $n\ge3$ and $\kappa>1$, and put
$q:=p_{n,n-1}(\kappa)=n/((n-1)\kappa+1)$.
Expanding \eqref{eq:P-prime-factorization} with $r=n-2$ and using
the beta-function recurrences gives
\begin{equation}\label{eq:radial-adjacent-comparison}
 \mathcal I_{n,n-2,\kappa}(q)
 =
 (n-2)B(1-q,n-2)
 \frac{
   (\kappa-1)\bigl(\kappa(n-2)^2-4\bigr)
 }{
   4\bigl(\kappa(n-1)^2-1\bigr)
 }.
\end{equation}
By \eqref{eq:I-sign}, $q\le p_{n,n-2}(\kappa)$ if and only if the
right-hand side is nonpositive.  For $n=3$, its sign is that of
$\kappa-4$, and $r=1$ is the only competing index.  For $n\ge4$,
it is strictly positive, proving
\eqref{eq:adjacent-branch-below-radial}.  The case $\kappa=1$
follows from the convention $p_{n,r}(1)=1$; for $n=2$, the radial
index is the only mixed negative index.
\end{proof}

This comparison concerns the kernel family; sharpness of the radial
upper bound in the remaining regimes is a separate question.

\begin{proposition}[Fixed-dimensional high-contrast behavior]\label{prop:fixed-dimensional-high-contrast}
Fix $n\ge2$ and $1\le r\le n-1$.  As $\kappa\to\infty$,
\begin{equation}\label{eq:fixed-n-asymptotic}
 p_{n,r}(\kappa)
 \sim
 \binom nr\frac{(n-r)^{n-r+1}}{r^{n-r}}\,\kappa^{-(n-r)}.
\end{equation}
Consequently,
\begin{equation}\label{eq:unrestricted-high-contrast}
 \underline p_n(\kappa)
 \sim n(n-1)^n\kappa^{-(n-1)}.
\end{equation}
\end{proposition}

\begin{proof}
Put
\[
 m:=n-r,
 \qquad
 p_\kappa:=p_{n,r}(\kappa),
\]
and define
\[
 J_\kappa(p):=
 \int_0^1z^{-p-1}
 \bigl(P_{n,r,\kappa}(z)-1\bigr)\dd z.
\]
Integration by parts gives
\begin{equation}\label{eq:I-J-identity}
 \mathcal I_{n,r,\kappa}(p)
 =pJ_\kappa(p)-1
 \qquad(0<p<1).
\end{equation}
In particular, at the critical exponent,
\begin{equation}\label{eq:critical-J-identity}
 1=p_\kappa J_\kappa(p_\kappa).
\end{equation}

For $p$ in any compact subinterval of $[0,1)$, expansion in powers of
$\kappa$ gives, uniformly in $p$,
\begin{equation}\label{eq:J-high-contrast-expansion}
 J_\kappa(p)
 =
 \left(\frac{\kappa r}{m}\right)^m
 B(m-p,r+1)+O(\kappa^{m-1}).
\end{equation}
Fix $q\in(0,1)$.  By \eqref{eq:J-high-contrast-expansion},
\[
 J_\kappa(q)\longrightarrow+\infty,
\]
and hence \eqref{eq:I-J-identity} gives
\[
 \mathcal I_{n,r,\kappa}(q)>0
\]
for all sufficiently large $\kappa$.  The sign characterization in
\Cref{lem:critical-exponent} then implies
\[
 p_\kappa<q.
\]
Since $q\in(0,1)$ was arbitrary,
\begin{equation}\label{eq:critical-root-tends-zero}
 p_\kappa\longrightarrow0.
\end{equation}

We may therefore substitute $p=p_\kappa$ into the uniform expansion
\eqref{eq:J-high-contrast-expansion}.  Using
\eqref{eq:critical-J-identity} and
\eqref{eq:critical-root-tends-zero}, we obtain
\begin{align*}
 p_\kappa
 &\sim
 \left(\frac{m}{\kappa r}\right)^m
 \frac{1}{B(m,r+1)}\\
 &=
 \binom nr
 \frac{m^{m+1}}{r^m}\,\kappa^{-m},
\end{align*}
because
\[
 B(m,r+1)=\frac{1}{m\binom nr}.
\]
Recalling that $m=n-r$ proves
\eqref{eq:fixed-n-asymptotic}.

Finally, the powers $\kappa^{-(n-r)}$ are distinct.  The smallest branch
for sufficiently large $\kappa$ is therefore $r=1$, for which
$m=n-1$.  Hence
\[
 \underline p_n(\kappa)
 \sim n(n-1)^n\kappa^{-(n-1)},
\]
which is \eqref{eq:unrestricted-high-contrast}.
\end{proof}

For comparison, the lower estimate from \cite[Theorem~1]{NascimentoTeixeiraJMPA} is
\begin{equation}\label{eq:old-lower-bound}
 L_n^{\mathrm{NT}}(\kappa)
 :=\frac{\mu_n}{\log n}
 \left(1+\left(1-\frac1\kappa\right)
 \left(1-\frac1n\right)\right)^{n-1}
 \kappa^{-(n-1)},
\end{equation}
where $\mu_n\in(1/4,1)$ is the explicit dimensional constant appearing there.  Hence, for fixed $n\ge3$,
\[
 L_n^{\mathrm{NT}}(\kappa)
 \sim
 \frac{\mu_n}{\log n}\left(2-\frac1n\right)^{n-1}
 \kappa^{-(n-1)},
\]
whereas \eqref{eq:unrestricted-high-contrast} gives
\[
 \underline p_n(\kappa)
 \sim n(n-1)^n\kappa^{-(n-1)}.
\]
Thus the present method improves the leading dimensional constant, but not the high-contrast power.  We include this consequence to clarify the reach of the general kernel family; no sharpness is asserted for $n\ge4$.


\begin{thebibliography}{99}

\bibitem{ArmstrongSilvestreSmart}
S.~N. Armstrong, L.~Silvestre, and C.~K. Smart,
\emph{Partial regularity of solutions of fully nonlinear, uniformly elliptic equations},
Comm. Pure Appl. Math. \textbf{65} (2012), no.~8, 1169--1184.

\bibitem{Caffarelli}
L.~A. Caffarelli,
\emph{Interior a priori estimates for solutions of fully nonlinear equations},
Ann. of Math. (2) \textbf{130} (1989), no.~1, 189--213.

\bibitem{CaffarelliCabre}
L.~A. Caffarelli and X.~Cabr\'e,
\emph{Fully nonlinear elliptic equations},
American Mathematical Society Colloquium Publications, vol.~43,
American Mathematical Society, Providence, RI, 1995.

\bibitem{CIL}
M.~G. Crandall, H.~Ishii, and P.-L. Lions,
\emph{User's guide to viscosity solutions of second order partial differential equations},
Bull. Amer. Math. Soc. (N.S.) \textbf{27} (1992), no.~1, 1--67.

\bibitem{EvansGariepy}
L.~C. Evans and R.~F. Gariepy,
\emph{Measure theory and fine properties of functions},
Textbooks in Mathematics, revised ed.,
CRC Press, Boca Raton, FL, 2015.


\bibitem{Le}
N.~Q. Le,
\emph{Polynomial decay in $W^{2,\varepsilon}$ estimates for viscosity supersolutions of fully nonlinear elliptic equations},
Math. Res. Lett. \textbf{27} (2020), no.~1, 189--207.

\bibitem{Lin}
F.-H. Lin,
\emph{Second derivative $L^p$-estimates for elliptic equations of nondivergent type},
Proc. Amer. Math. Soc. \textbf{96} (1986), no.~3, 447--451.

\bibitem{Mooney}
C.~Mooney,
\emph{A proof of the Krylov--Safonov theorem without localization},
Comm. Partial Differential Equations \textbf{44} (2019), no.~8, 681--690.

\bibitem{NascimentoTeixeiraJMPA}
T.~M. Nascimento and E.~V. Teixeira,
\emph{New regularity estimates for fully nonlinear elliptic equations},
J. Math. Pures Appl. (9) \textbf{171} (2023), 1--25.

\bibitem{NascimentoTeixeiraJDE}
T.~M. Nascimento and E.~V. Teixeira,
\emph{On the sharp Hessian integrability conjecture in the plane},
J. Differential Equations \textbf{414} (2025), 890--903.

\end{thebibliography}
\end{document}